\documentclass[12pt, reqno]{amsart}
\usepackage{amsmath}
\usepackage{amssymb}
\usepackage{epsfig}
\usepackage{graphicx}
\usepackage{color}
\usepackage{fullpage}
\usepackage{comment}
\definecolor{shadecolor}{gray}{0.875}
\usepackage{amscd}
\usepackage{stmaryrd}
\usepackage{amsmath}
\usepackage{ulem}
\usepackage{amsthm}
\usepackage{mathrsfs}
\usepackage{array}
\usepackage{booktabs}

\usepackage{csquotes}
\usepackage{enumitem}
\usepackage{graphics}
\usepackage{tikz-cd}

\usepackage{array}
\usepackage{booktabs}
\usepackage{makecell}
\usepackage{threeparttable}
\usepackage{longtable}

\usepackage{xcolor}
\definecolor{darkred}{RGB}{120,0,0}
\definecolor{darkgreen}{RGB}{0,90,45}

\usepackage[pagebackref]{hyperref}
\hypersetup{
	colorlinks=true,
	linkcolor=darkred,    % 定理、章节、公式等交叉引用
	citecolor=darkgreen,  % 文献引用
	urlcolor=darkred      % 网址
}

\usepackage{textcomp}

\let\oldtocsection=\tocsection

\let\oldtocsubsection=\tocsubsection

\let\oldtocsubsubsection=\tocsubsubsection

\renewcommand{\tocsection}[2]{\hspace{0em}\oldtocsection{#1}{#2}}
\renewcommand{\tocsubsection}[2]{\hspace{1em}\oldtocsubsection{#1}{#2}}
\renewcommand{\tocsubsubsection}[2]{\hspace{2em}\oldtocsubsubsection{#1}{#2}}
\makeatletter

\newcommand{\Rmnum}[1]{\expandafter\@slowromancap\romannumeral #1@}
\makeatother
\numberwithin{equation}{section}

\input xy
\xyoption{all}

\calclayout
\allowdisplaybreaks[3]

\theoremstyle{plain}
\newtheorem{prop}{Proposition}[section]

\newtheorem{coro}[prop]{Corollary}

\newtheorem{lemm}[prop]{Lemma}

\theoremstyle{definition}
\newtheorem{defi}[prop]{Definition}

\newtheorem{rema}[prop]{Remark}

\newtheorem{exam}[prop]{Example}

\newtheorem{maintheorem}{Theorem}
\newtheorem{maincorollary}[maintheorem]{Corollary}

\def\ra{\rightarrow}

\def\bC{{\mathbb C}}

\def\bP{{\mathbb P}}

\def\bR{{\mathbb R}}
\def\bZ{{\mathbb Z}}

\def\Pic{\mathrm{Pic}}

\def\Pic{\mathrm{Pic}}

\newcommand{\defterm}[1]{\textit{#1}}

\makeatother
\makeatletter

\author{Runxuan Gao}
\address{Graduate School of Mathematics, Nagoya University, Furocho Chikusa-ku, Nagoya, 464-8602, Japan}
\email{m20015x@math.nagoya-u.ac.jp}

\title[Generalized Iskovskikh surfaces]{Generalized Iskovskikh surfaces and unirationality of del Pezzo surfaces of degree $2$}

\begin{document}
\date{\today}
\begin{abstract}
	We introduce generalized Iskovskikh surfaces and study their relation to quasi-\'etale double covers. As an application, we prove that a del Pezzo surface of degree $2$ over an infinite field of characteristic different from $2$ is unirational if it has a rational generalized Eckardt point. This settles a case left open by Manin's construction. We also obtain unirationality results for Du Val del Pezzo surfaces of degrees $1$ and $2$, and discuss applications to the Hilbert property, equivariant unirationality, and the real plane Cremona group.
\end{abstract}
\maketitle
\setcounter{tocdepth}{2}
%\tableofcontents

\section*{Introduction}
A variety $X$ defined over a field $k$ is \defterm{unirational}
if there is a dominant rational map $\bP^n \dasharrow X$
defined over $k$ for some $n$.
It is \defterm{rational} if it is birational over $k$ to $\bP^n$.
A \defterm{del Pezzo surface} $S$ is a smooth projective surface
whose anticanonical line bundle is ample.
Its \defterm{degree} is $d = (-K_S)^2$.
Every del Pezzo surface is geometrically rational,
but it need not be rational over its field of definition.

For a smooth projective variety, the existence of a $k$-rational
point is a necessary condition for $k$-(uni)rationality.
It is therefore natural to ask whether this condition is
sufficient for del Pezzo surfaces.
The study of this question goes back to the classical work
of Manin \cite{manin1986cubicforms}.
It is known that a del Pezzo surface $S$ with
$S(k)\neq\varnothing$ is $k$-rational if $d\geq 5$,
and $k$-unirational if $d\geq 3$.
For $d\leq 4$, such a surface need not be $k$-rational.
However, we still do not know any example of a
del Pezzo surface with a $k$-rational point
that is not $k$-unirational.

The unirationality problem for del Pezzo surfaces of degrees $1$ and $2$
has not yet been completely resolved.
For recent progress in degree $1$,
see \cite{cheltsov2026unirational}.
In degree $2$, Manin proved unirationality under suitable
conditions on a $k$-rational point.
Let $\pi:S\to\bP^2$ be the double cover induced by $|-K_S|$,
and let $R$ be its ramification divisor.
A geometric point of $S$ is called a
\defterm{generalized Eckardt point} if it lies on four
$(-1)$-curves.
Salgado, Testa, and V\'arilly-Alvarado \cite[Corollary~3.3]{STVA14}
show that Manin's construction
proves the $k$-unirationality of $S$ whenever $S$ has a
$k$-rational point lying outside $R$ that is not a
generalized Eckardt point.
This leaves two types of points untreated:
points on $R$ and generalized Eckardt points.

In this paper, we resolve the generalized Eckardt point case
over infinite fields of characteristic different from $2$.

\begin{maintheorem}\label{theo:main}
	Let $S$ be a del Pezzo surface of degree $2$ over an infinite
	field $k$ of characteristic different from $2$.
	If $S$ has a $k$-rational point lying on four $(-1)$-curves
	of $S_{\overline{k}}$, then $S$ is unirational over $k$.
\end{maintheorem}

Combining Theorem~\ref{theo:main} with the known
unirationality results \cite[Corollary~3.3]{STVA14} and \cite[Theorem~1.1]{FvL16}, we obtain the following corollary.

\begin{maincorollary}\label{coro:unirationality}
	Let $S$ be a smooth del Pezzo surface of degree $2$
	over a field $k$.
	Suppose that either
		$k$ is finite, or
		$k$ is infinite of characteristic different from $2$
		and $(S\setminus R)(k)\neq\varnothing$, where $R$
		is the ramification divisor of the anticanonical morphism.
	Then $S$ is $k$-unirational.
\end{maincorollary}

The following corollary removes the restriction on generalized
Eckardt points in \cite[Theorem~1.2(2)]{Tamanoi}.

\begin{maincorollary}\label{coro:rdp}
	Let $X$ be a Du Val del Pezzo surface of degree~$2$ over an
	infinite field $k$ of characteristic different from~$2$.
	Let $R$ be the ramification divisor of the anticanonical morphism.
	Then $X$ is $k$-unirational if and only if
	$(X\backslash R)(k)\ne\varnothing$.
\end{maincorollary}

We record four applications of our results and construction.
In the first three, let $S$ be a smooth del Pezzo surface of degree~$2$
over an infinite field $k$ of characteristic different from~$2$,
and assume that $(S\setminus R)(k)\ne\varnothing$.

\begin{enumerate}
	\item In characteristic zero, the conclusions of
	\cite[Theorems~1.1 and~1.3]{DSW2024} hold under the
	hypothesis $(S\backslash R)(k)\ne\varnothing$.
	Thus $S$ is Hilbert-unirational and has the Hilbert property
	whenever $k$ is Hilbertian. Over a number field, $S$ also
	satisfies weak weak approximation.
	
	\item If $S$ has a $k$-rational generalized Eckardt point,
	then \eqref{diag} exhibits $S$ as $k$-birational to the quotient
	of a smooth del Pezzo surface of degree~$2$ by a $k$-involution.
	This answers \cite[Question~1.20]{Tre} affirmatively in the degree-$2$ case.
	
	\item If a finite group $G$ acts on $S$ over $k$ and fixes a point
	of $(S\setminus R)(k)$, then $S$ is $G$-unirational
	(Proposition~\ref{prop:equivariant-unirationality}).
	To our knowledge, this is the first proof of \(G\)-unirationality for all the actions in Case \((2\mathrm B)\) of Dolgachev and Duncan
	(see Remark~\ref{rema:noncyclic}).
	
	\item The fixed curve of a real Kowalevskaya involution
	occurs in uncountably many conjugacy classes, answering
	\cite[Question~6.9]{CMYZ26} in the negative (Proposition~\ref{prop:kowalevskaya}).
\end{enumerate}

As another application, we record a degree-one
version of \cite[Theorem~1.2]{Tamanoi},
which extends \cite[Proposition~A.1(i)]{prokhorov2026}.
\begin{maincorollary}\label{coro:rdpdp1}
	Let $S$ be a Du Val del Pezzo surface of degree~$1$
	over a perfect field $k$ of characteristic different from~$2$.
	Suppose that $S$ is singular.
	If the singularity type of $S_{\bar k}$ is neither $A_1$ nor $A_2$,
	then $S$ is $k$-unirational.
\end{maincorollary}
In particular, every generalized Iskovskikh surface of degree~$1$
over $k$ is $k$-unirational.
In fact, proving $k$-unirationality for all surfaces of types $A_1$ and $A_2$
in degree~$1$ is equivalent to settling the remaining case of the unirationality
problem for smooth del Pezzo surfaces of degree~$2$.

Our construction is based on a generalization of
classical Iskovskikh surfaces.
For an intersection of two quadrics in $\bP^4$
with exactly two singular points, both of type $A_1$,
the line joining them may or may not lie on the surface.
If this line is not contained in the surface and
the two singular points are Galois conjugate,
the surface is called an Iskovskikh surface
\cite[p.~74]{CorayTsfasman1988}.
%Over a perfect field, these are the only singular
%Du Val del Pezzo surfaces of degree $4$ with a smooth
%$k$-rational point that can fail to be $k$-rational;
%see \cite[Theorem~7.2]{CorayTsfasman1988}.

We generalize this distinction to Du Val del Pezzo
surfaces of degree at most $4$ whose degree and
singularity type do not uniquely determine the
configuration of their $(-1)$-curves.
These generalized Iskovskikh surfaces
(Definition~\ref{defi:isk}) come in two geometric types
for each degree and singularity type
(Proposition~\ref{prop:two-embeddings}).
They are distinguished by the existence of a
$(-1)$-curve passing through certain singular points in a certain way (Proposition~\ref{prop:distinguished-curve}).
We call the surface of type $+$ when this
distinguished curve exists, and of type $-$ otherwise.
When $\operatorname{char}k\neq2$, a generalized
Iskovskikh surface of degree $1$ is of type $+$
if and only if it admits a quasi-\'etale double
cover over $k$
(Lemma~\ref{lemm:torsion} and
Corollary~\ref{coro:cover}).

To prove Theorem~\ref{theo:main}, we blow up the given
generalized Eckardt point and contract the strict
transforms of the four $(-1)$-curves.
The resulting surface $S_2$ is a generalized Iskovskikh
surface of degree $1$ and type $+$.
It therefore admits a quasi-\'etale double cover $\pi:T\to S_2$
by a smooth del Pezzo surface $T$ of degree $2$.
After a suitable quadratic twist of $\pi$, the surface $T$ admits a
$k$-rational point to which Manin's construction applies.
This construction is summarized in \eqref{diag}.

It is worth mentioning that over $\bC$,
the surface $T$ also admits generalized Eckardt points.
The anticanonical branch quartics of $S$ and $T$
are related by Pantazis's bigonal duality
(Proposition~\ref{prop:inverse-construction}).
Thus our construction provides a geometric realization
of the bigonal duality in terms of degree-two del Pezzo surfaces.

The paper is organized as follows.
Section~\ref{sec:1} studies generalized Iskovskikh surfaces
and their quasi-\'etale double covers.
Section~\ref{sec:2} proves Theorem~\ref{theo:main} and its
corollaries. Section~\ref{sec:3} treats equivariant unirationality,
identifies the construction in \eqref{diag} with Pantazis's bigonal
duality on branch quartics, and applies the quotient construction
to real Kowalevskaya involutions.

\noindent\textbf{Acknowledgments.}
The author thanks Colliot-Th{\'e}l{\`e}ne
for pointing out \cite{CorayTsfasman1988}
at the conference Birational Geometry and Number Theory,
held in Edinburgh in 2024,
and De-Qi Zhang for answering his question on quasi-\'etale covers
and referring him to \cite{EsnaultViehweg}
at the conference Birational Geometry and Algebraic Dynamics,
held at the University of Tokyo in November 2023.
He also thanks Hiroto Akaike, from whom he learned of
Kowalevskaya's result during a conversation.
The author was partially supported by a Grant-in-Aid
for JSPS Fellows (Grant No.~24KJ1234).

\noindent\textbf{AI disclosure.}
The notion of generalized Iskovskikh surfaces and the proof
of Theorem~\ref{theo:main} were developed and recorded in
the author's private notes during his work on
\cite{gao2024quasi}, without the use of AI tools.
In preparing this paper, the author used ChatGPT-6 Astra
to help locate references, proofread the manuscript,
and improve the language.
The paper itself was written by the author
who is responsible for its correctness.

\section{Generalized Iskovskikh surfaces}\label{sec:1}

%Throughout the paper, $k$ is a field of characteristic
%different from $2$.
A \defterm{Du Val del Pezzo surface} over a field $k$ is
a projective geometrically integral surface $S$
with ample anticanonical divisor such that
$S_{\bar k}$ has only Du Val singularities.

Let $S$ be a Du Val del Pezzo surface of degree
$1\le d\le6$ over $k$, and let
$\rho:X\to S_{\bar k}$ be the minimal resolution.
Then $X$ is a weak del Pezzo surface obtained by successively
blowing up $9-d$ possibly infinitely near points of $\mathbb P^2$.
The Picard group of $X$ is
\[
\operatorname{Pic}(X)
= \mathbb Z h \oplus \bigoplus_{i=1}^{9-d}\mathbb Z e_i,
\]
where $h$ is the pullback of the class of a line and the $e_i$
are the total transforms of the exceptional divisor classes.
These classes satisfy
\begin{equation}\label{eq:basis}
	h^2=1,\qquad h\cdot e_i=0,\qquad
	e_i\cdot e_j=-\delta_{ij},
	\qquad K_X=-3h+\sum_{i=1}^{9-d}e_i.
\end{equation}
A class $\alpha\in\operatorname{Pic}(X)$ is called a
$(-2)$-class or a root, if
\[
\alpha^2=-2,\qquad \alpha\cdot K_X=0.
\]
The roots form a root system whose type is listed in
Table~\ref{tab:del-pezzo-root-systems}.
The lattice generated by these roots, equipped with the intersection
pairing, depends only on $d$ up to isometry.
We call it the
\defterm{del Pezzo lattice} of degree $d$.
For $1\le d\le6$, this lattice coincides with
$K_X^\perp\subset\operatorname{Pic}(X)$.

\begin{table}[htbp]
	\caption{Root systems associated with weak del Pezzo surfaces.}
	\label{tab:del-pezzo-root-systems}
	\centering
	\begin{tabular}{lcccccc}
		\toprule
		Degree $d$
		& $1$ & $2$ & $3$ & $4$ & $5$ & $6$  \\
		\midrule
		Root system
		& $E_8$ & $E_7$ & $E_6$ & $D_5$ & $A_4$
		& $A_2\oplus A_1$\\
		\bottomrule
	\end{tabular}
\end{table}

Let $C_1,\ldots,C_r$ be the irreducible $\rho$-exceptional curves.
We say a root $\alpha$ is \defterm{effective} if and only if
\[
\alpha=\sum_{j=1}^{r}n_j[C_j],
\qquad n_j\in\mathbb Z_{\geq 0}.
\]
An effective root is represented by an irreducible $(-2)$-curve
if and only if it equals $[C_j]$ for some $j$.
Let $E$ be a del Pezzo lattice.
The Weyl group $W(E)$ acts on $E$.
The following proposition classifies root sublattices of $E$
which are isometric, but do not belong to the same $W(E)$-orbit.
It can be either read off from the classification
results in \cite[Section~4]{DynkinMinchenko2010}, or derived from the general theory
developed there.

\begin{prop}[{\cite{DynkinMinchenko2010}}]
	\label{prop:two-embeddings}
	Let $E$ be a del Pezzo lattice of degree $1\leq d\leq 6$.
	A root lattice $R$ admits two isometric embeddings into $E$
	whose images are not $W(E)$-conjugate if and only if
	\[
	R\cong\bigoplus_{j=1}^{s}A_{2m_j-1},\qquad m_j\ge1,
	\]
	such that $q=\sum_j m_j$ satisfies one of the following conditions:
	\begin{enumerate}
		\item $q=4$ when $d=1$;
		\item either $q=3$, or $q=4$ and $\min_j m_j=1$, when $d=2$;
		\item $q=2$ when $d=4$;
		\item $R=A_1$ when $d=6$;
	\end{enumerate}
	In particular, no such $R$ exists for $d=3$ or $5$.
\end{prop}

Define the following roots under the standard basis:
\[
	\alpha_i=e_i-e_{i+1}
	\quad(1\le i\le 8-d),\quad
	u=h-e_1-e_2-e_3,
	\]
	\[
	v=2h-\sum_{i=1}^{6}e_i\quad(d=1),
	\qquad
	w=2h-\sum_{i=2}^{7}e_i\quad(d=2).
\]
For each type $R$ satisfying the conditions of Proposition~\ref{prop:two-embeddings},
representatives of the two Weyl orbits are listed in Tables~\ref{tab:root-embeddings-degree1}, \ref{tab:root-embeddings-degree2},
\ref{tab:root-embeddings-degree4},
and \ref{tab:root-embeddings-degree6}.

\begin{table}[htbp]
	\caption{Representatives of the two Weyl orbits in degree $1$.}
	\label{tab:root-embeddings-degree1}
	\centering
	\small
	\begin{tabular}{lll}
		\toprule
		Type $R$ & $\Delta_+$ & $\Delta_-$ \\
		\midrule
		$A_7$
		& $\{\alpha_1,\ldots,\alpha_6,v\}$
		& $\{\alpha_1,\ldots,\alpha_7\}$ \\[2pt]
		$A_5\oplus A_1$
		& $\{\alpha_1,\ldots,\alpha_5,v\}$
		& $\{\alpha_1,\ldots,\alpha_5,\alpha_7\}$ \\[2pt]
		$2A_3$
		& $\{\alpha_1,\alpha_2,\alpha_3,\alpha_5,\alpha_6,v\}$
		& $\{\alpha_1,\alpha_2,\alpha_3,\alpha_5,\alpha_6,\alpha_7\}$ \\[2pt]
		$A_3\oplus2A_1$
		& $\{\alpha_1,\alpha_2,\alpha_3,\alpha_5,v\}$
		& $\{\alpha_1,\alpha_2,\alpha_3,\alpha_5,\alpha_7\}$ \\[2pt]
		$4A_1$
		& $\{\alpha_1,\alpha_3,\alpha_5,v\}$
		& $\{\alpha_1,\alpha_3,\alpha_5,\alpha_7\}$ \\
		\bottomrule
	\end{tabular}
\end{table}

\begin{table}[htbp]
	\caption{Representatives of the two Weyl orbits in degree $2$.}
	\label{tab:root-embeddings-degree2}
	\centering
	\small
	\begin{tabular}{lll}
		\toprule
		Type $R$ & $\Delta_+$ & $\Delta_-$ \\
		\midrule
		$A_5$
		& $\{u,\alpha_3,\alpha_4,\alpha_5,\alpha_6\}$
		& $\{\alpha_2,\alpha_3,\alpha_4,\alpha_5,\alpha_6\}$ \\[2pt]
		$A_3\oplus A_1$
		& $\{u,\alpha_3,\alpha_4,\alpha_6\}$
		& $\{\alpha_2,\alpha_3,\alpha_4,\alpha_6\}$ \\[2pt]
		$3A_1$
		& $\{u,\alpha_4,\alpha_6\}$
		& $\{\alpha_2,\alpha_4,\alpha_6\}$ \\[2pt]
		$A_5\oplus A_1$
		& $\{\alpha_1,u,\alpha_3,\alpha_4,\alpha_5,\alpha_6\}$
		& $\{\alpha_2,\alpha_3,\alpha_4,\alpha_5,\alpha_6,w\}$ \\[2pt]
		$A_3\oplus2A_1$
		& $\{\alpha_1,u,\alpha_3,\alpha_4,\alpha_6\}$
		& $\{\alpha_2,\alpha_3,\alpha_4,\alpha_6,w\}$ \\[2pt]
		$4A_1$
		& $\{\alpha_1,u,\alpha_4,\alpha_6\}$
		& $\{\alpha_2,\alpha_4,\alpha_6,w\}$ \\
		\bottomrule
	\end{tabular}
\end{table}

\begin{table}[htbp]
	\caption{Representatives of the two Weyl orbits in degree $4$.}
	\label{tab:root-embeddings-degree4}
	\centering
	\small
	\begin{tabular}{lll}
		\toprule
		Type $R$ & $\Delta_+$ & $\Delta_-$ \\
		\midrule
		$A_3$
		& $\{\alpha_1,\alpha_2,\alpha_3\}$
		& $\{u,\alpha_3,\alpha_4\}$ \\[2pt]
		$2A_1$
		& $\{\alpha_1,\alpha_3\}$
		& $\{u,\alpha_4\}$ \\
		\bottomrule
	\end{tabular}
\end{table}
\begin{table}[htbp]
	\caption{Representatives of the two Weyl orbits in degree $6$.}
	\label{tab:root-embeddings-degree6}
	\centering
	\begin{tabular}{lll}
		\toprule
		Type $R$ & $\Delta_+$ & $\Delta_-$ \\
		\midrule
		$A_1$ & $\{\alpha_1\}$ & $\{u\}$ \\
		\bottomrule
	\end{tabular}
\end{table}

By a \((-1)\)-curve on a Du Val del Pezzo surface we mean the image of a \((-1)\)-curve on its geometric minimal resolution.
\begin{prop}\label{prop:distinguished-curve}
	Assume that $d\in\{1,2,4\}$ and $S$ is a Du Val del Pezzo surface defined over a field $k$,
	and that the induced lattice satisfies the equivalent conditions of
	Proposition~\ref{prop:two-embeddings}.
	Then $S_{\overline{k}}$ corresponds to the $+$ model if and only if
	$S_{\overline{k}}$ contains a $(-1)$-curve $C$ whose strict transform meets
	the central component
	of every exceptional chain above every singularity, except that one $A_1$ singularity is omitted
	when $(d,q)=(2,4)$.
	Moreover, such a curve $C$ is unique if it exists.
\end{prop}
\begin{proof}
	Let $\rho:X\to S_{\bar k}$ be the minimal resolution.
	By Proposition~\ref{prop:two-embeddings},
	the $\rho$-exceptional locus is a disjoint
	union of chains of $(-2)$-curves.
	Write $K=K_X$ and $L=\operatorname{Pic}(X)$.
	
	We first assume that $d=1$, or that $(d,q)=(2,3)$.
	Write $D_{j,1},\ldots,D_{j,2m_j-1}$
	for the components of the $j$-th exceptional chain,
	ordered along the chain.
	For any subset $I$ of the set of
	exceptional chains, put
	\begin{equation}\label{eq:Z}
		Z=-\frac Kd-\frac12\sum_{j\in I}
		\sum_{i=1}^{2m_j-1}\min\{i,2m_j-i\}D_{j,i}.
	\end{equation}
	A direct computation gives
	\[
	K\cdot Z=-1,
	\qquad
	Z^2=\frac1d-\frac12\sum_{j\in I}m_j.
	\]
	Take $I$ to be the set of all exceptional chains.
	Then $\sum_{j\in I}m_j=2+2/d$, so
	\begin{equation}\label{eq:Z2}
		Z^2=K\cdot Z=-1.
	\end{equation}
	Moreover, $Z$ intersects each central component $D_{j,m_j}$
	for $j\in I$
	with multiplicity one and every other
	component with multiplicity zero.

	We now show that
	any $(-1)$-curve meeting all
	central components of $I$ must have numerical class $Z$.
	Let $F$ be such a $(-1)$-curve. Then
	\[
	F\cdot Z\le
	\frac1d-\frac12\sum_{j\in I}m_j=-1.
	\]
	Hence $F-Z\in K^\perp$ and
	\[
	(F-Z)^2=-2-2F\cdot Z\ge0.
	\]
	Negative definiteness of $K^\perp$ forces $[F]=Z$.
	
	We now show that $Z$ represents a $(-1)$-curve if and only
	if $X$ is defined by the $+$ model.
	In view of \eqref{eq:Z2}, and since $Z$ has non-negative intersection with every $(-2)$-curve
	on $X$, it suffices to prove that $Z\in L$ by
	\cite[Corollaire au Th\'eor\`eme~III.2]{Demazure}
	(see also \cite[Lemma~3.8]{DerenthalJoyceTeitler}).
	
	Set $B_I=\sum_{j\in I}\sum_{i\ {\rm odd}}D_{j,i}$.
	Since $\min\{i,2m_j-i\}\equiv i\pmod2$, we have
	\[
	Z\in L\iff \frac2dK+B_I\in2L.
	\]
	One checks directly that
	\[
	\frac2dK+B_I\equiv
	\begin{cases}
		0,&\text{in the $+$ model},\\
		e_1+\cdots+e_8,
		&\text{in the $-$ model with $d=1$},\\
		h+e_1,
		&\text{in the $-$ model with $d=2$}
	\end{cases}
	\pmod{2L},
	\]
	which shows that $Z$ is a $(-1)$-curve if and only
	if $X$ is defined by the $+$ model.

	Now suppose that $(d,q)=(2,4)$, and define $Z$ by
	\eqref{eq:Z} after omitting an $A_1$ component $D_0$.
	The same argument applies, with
	\[
	Z\in L\iff K+B-D_0\in2L,
	\qquad\text{where }
	B=\sum_j\sum_{i\text{ odd}}D_{j,i}.
	\]
	In the $+$ model, a direct computation shows that $K+B\equiv\alpha_1\pmod{2L}$,
	so the only $A_1$ component giving an integral class is $D_0=\alpha_1$.
	In the $-$ model, every exceptional component has
	even $h$-coefficient, whereas $K$ has odd $h$-coefficient.
	Thus no choice of $D_0$ gives an integral class.
	
	Finally for $d=4$, there are exactly $16$ $(-1)$-classes,
	and \cite[Appendix A]{CheltsovProkhorov2021}
	lists the dual graphs of negative curves involved.
	One checks directly that
	$h-e_1-e_2$ for $A_3$
	and $h-e_1-e_3$ for $2A_1$
	represent the unique $(-1)$-curves on the $+$ model
	intersecting the central components,
	and that no $-$ model admits such a curve.
\end{proof}
\begin{rema}
	See \cite[Theorem~4.6 and Remark~4.1]{DynkinMinchenko2010}
	for an explanation of the two possible values of $q$
	when $d=2$.
\end{rema}

It is natural to introduce the following notion to distinguish
the two kinds of surfaces with the same singularity type.
\begin{defi}\label{defi:isk}
	Let $S$ be a Du Val del Pezzo surface of degree
	$d\in\{1,2,4\}$ over a field $k$, and let
	$\rho:X\to S_{\bar k}$ be its minimal resolution.
	Let $R(S)\subset K_X^\perp$ be the sublattice generated
	by the classes of the irreducible $\rho$-exceptional curves.
	
	For $\varepsilon\in\{-,+\}$, we call $S$ a
	\defterm{generalized Iskovskikh surface of type $\varepsilon$}
	if, under a marking of $\Pic(X)$ as in \eqref{eq:basis},
	the sublattice $R(S)$ is $W(E)$-conjugate to
	$\mathbb Z\Delta_\varepsilon$ for one of the configurations
	of degree $d$ listed in
	Tables~\ref{tab:root-embeddings-degree1} to
	\ref{tab:root-embeddings-degree4}.
\end{defi}

In other words, among Du Val del Pezzo surfaces of degree
at most $4$, generalized Iskovskikh surfaces are precisely
those whose degree and singularity type do not uniquely
determine the configuration of their geometric $(-1)$-curves.
Such a surface $S$ is of type $+$ if and only if
$S_{\bar k}$ contains
a unique $(-1)$-curve with the incidence properties described
in Proposition~\ref{prop:distinguished-curve}.
We call this curve the \defterm{distinguished curve}.
Generalized Iskovskikh surfaces of type $-$ with degree $4$
and two Galois-conjugate $A_1$ singularities are precisely
the Iskovskikh surfaces introduced in
\cite[p.~74]{CorayTsfasman1988}.

\begin{rema}
	Singular del Pezzo surfaces of degree $6$
	with one $A_1$ singularity
	also give rise to two geometric types, as in
	Table~\ref{tab:root-embeddings-degree6}.
	But the characterization by a unique distinguished curve
	does not extend to this case.
\end{rema}

Throughout this paper, a \defterm{quasi-\'etale cover}
is a finite surjective morphism between normal geometrically integral varieties that is \'etale in codimension one.
While working on \cite{gao2024quasi}, we observed that, for generalized Iskovskikh surfaces,
one of the models tends to admit a nontrivial quasi-\'etale cover.
The following lemmas and corollary explain this phenomenon.
\begin{lemm}\label{lemm:A-chain-torsion}
	Let $R=\bigoplus_{j=1}^s A_{2m_j-1}$, with simple roots
	$\alpha_{j,1},\ldots,\alpha_{j,2m_j-1}$ ordered along each chain.
	Then $(R^\vee/R)[2]\cong(\mathbb Z/2\mathbb Z)^s$,
	with basis given by the classes of
	\begin{equation}\label{eq:basis-torsion}
		\frac12\sum_{i\text{ odd}}\alpha_{j,i},
		\qquad 1\le j\le s.
	\end{equation}
\end{lemm}

\begin{proof}
	Let $y\in R^\vee$ satisfy $2y\in R$, and write
	\[
	2y=\sum_{j,i}a_{j,i}\alpha_{j,i},
	\qquad a_{j,i}\in\mathbb Z.
	\]
	Set $a_{j,0}=a_{j,2m_j}=0$.
	The integrality of $y\cdot\alpha_{j,i}$ gives
	\[
	a_{j,i-1}\equiv a_{j,i+1}\pmod2.
	\]
	Thus, on each chain, the even-indexed coefficients are even
	and the odd-indexed coefficients have the same parity.
	Hence the classes in \eqref{eq:basis-torsion} generate $(R^\vee/R)[2]$.
	They form a basis since their classes are linearly independent over $\mathbb F_2$.
\end{proof}
\begin{lemm}\label{lemm:torsion}
	Assume the equivalent conditions of
	Proposition~\ref{prop:two-embeddings} hold.
	Then the following are equivalent:
	\begin{enumerate}
		\item At least one of the two embeddings satisfies
		$(E/R)[2]\neq0$.
		\item The Dynkin graph of $R$ contains four pairwise
		nonadjacent vertices.
	\end{enumerate}
	Moreover, when these conditions hold, exactly one of the two
	embeddings satisfies $(E/R)[2]=\bZ/2\bZ$.
	This occurs for the $+$ models in degree $1$
	and the $-$ models in degree $2$.
\end{lemm}
\begin{proof}
	By Proposition~\ref{prop:two-embeddings}, we have
	\[
	R=\bigoplus_{j=1}^s A_{2m_j-1},
	\qquad q=\sum_jm_j\le4.
	\]
	The maximal number of pairwise nonadjacent vertices
	in the Dynkin graph of $R$ is $q$.
	Thus condition~(2) is equivalent to $q=4$.
	
	Write $D_{j,1},\ldots,D_{j,2m_j-1}$ for the
	components of the $j$-th chain, and put
	\(
	B_j=\sum_{i\ {\rm odd}}D_{j,i}.
	\)
	Let $x\in E$ represent a nonzero element of $(E/R)[2]$.
	Since $2x\in R$ and $E$ is integral, we have
	$x\in R^\vee$.
	By Lemma \ref{lemm:A-chain-torsion}, we have
	\[
	x\equiv\beta:=\frac12\sum_{j\in J}B_j\pmod R
	\]
	for some nonempty set $J$.
	Moreover, since $\beta$ belongs to the even lattice $E$
	and the roots occurring in $\beta$ are pairwise orthogonal,
	we have
	\[
	\beta^2=-\frac12\sum_{j\in J}m_j\in2\mathbb Z.
	\]
	Hence $4\mid\sum_{j\in J}m_j$.
	As $0<\sum_{j\in J}m_j\le q\le4$, we obtain $q=4$.
	This proves $(1)\Rightarrow(2)$.
	
	Conversely, suppose that $q=4$; then $d=1$ or $2$.
	Tables~\ref{tab:root-embeddings-degree1}
	and~\ref{tab:root-embeddings-degree2} give
	\[
	\begin{aligned}
		\frac{\alpha_1+\alpha_3+\alpha_5+v}{2}
		&=h-e_2-e_4-e_6
		&&\text{in the degree-$1$ $+$ model},\\
		\frac{\alpha_2+\alpha_4+\alpha_6+w}{2}
		&=h-e_3-e_5-e_7
		&&\text{in the degree-$2$ $-$ model}.
	\end{aligned}
	\]
	In either case, this defines a nonzero element of $(E/R)[2]$.

	The preceding argument also forces $J=\{1,\ldots,s\}$,
	so the only possible nonzero $2$-torsion class is
	represented by $\frac12\sum_jB_j$.
	This gives
	$(E/R)[2]\cong\bZ/2\bZ$ in the degree-$1$ $+$ model
	and the degree-$2$ $-$ model.
	In the degree-$1$ $-$ model and the degree-$2$ $+$ model
	with $q=4$, respectively, we have
	\[
	\sum_jB_j=\alpha_1+\alpha_3+\alpha_5+\alpha_7,
	\qquad
	\sum_jB_j=\alpha_1+u+\alpha_4+\alpha_6.
	\]
	Both have an odd coefficient in the standard basis,
	so they do not lie in $2E$.
	Hence the other embedding has trivial $2$-torsion.
\end{proof}

\begin{coro}\label{coro:cover}
	Assume $\operatorname{char}k\neq2$.
	Every generalized Iskovskikh surface of type $+$ and degree $1$ over $k$
	admits a quasi-\'etale double cover defined over $k$.
	The same holds for a generalized Iskovskikh surface
	of type $-$ and degree $2$ with $q=4$,
	provided its minimal regular resolution over $k$
	has a $k$-rational point.
\end{coro}

\begin{proof}
	Write $U=S_{sm}$ for the smooth part.
	Let $\rho:X\to S_{\bar k}$ be the minimal resolution,
	and identify $\rho^{-1}(U_{\bar k})$ with $U_{\bar k}$.
	On $X$, put
	\[
	\overline{B}
	=\sum_j\sum_{\substack{1\le i\le 2m_j-1\\ i\text{ odd}}}
	D_{j,i}.
	\]
	The Galois action permutes the exceptional chains
	and preserves the odd indices in the chains.
	Hence $\overline{B}$ is Galois invariant.
	It is a disjoint union of four $(-2)$-curves
	in the cases of $q=4$,
	and the calculation in Lemma~\ref{lemm:torsion}
	shows that $\mathcal O_X(\overline{B})$
	has a square root $\overline{\mathcal M}$.
	This square root is unique up to isomorphism, since
	$\operatorname{Pic}(X)$ is torsion-free,
	and is therefore Galois invariant.
	Put $\overline{\mathcal L}=\overline{\mathcal M}|_{U_{\bar k}}$.
	Then we have
	$\bar{\mathcal L}^{\otimes2}\cong\mathcal O_{U_{\bar k}}$.
	Moreover, $\overline{\mathcal L}$ is nontrivial.
	Indeed, the kernel of
	$\operatorname{Pic}(X)\to\operatorname{Pic}(U_{\bar k})$
	is the lattice $R$ generated by the $(-2)$-curves,
	whereas $[\overline{\mathcal M}]=[\overline B]/2\notin R$.
	
	Write $k^s$ for the separable closure of $k$.
	Since $S$ is geometrically normal and projective,
	and $S\setminus U$ has codimension two, we have
	$H^0(U_{k^s},\mathcal O_{U_{k^s}})=k^s$.
	As $\operatorname{char}k\neq2$, invariance of prime-to-characteristic
	torsion of Picard groups under purely inseparable extensions \cite[\href{https://stacks.math.columbia.edu/tag/0CD5}{Tag 0CD5}]{stacks-project} gives
	\[
	\operatorname{Pic}(U_{k^s})[2]
	\xrightarrow{\sim}
	\operatorname{Pic}(U_{\bar k})[2].
	\]
	Thus $\overline{\mathcal L}$ descends uniquely to a Galois-invariant
	class $[\mathcal L_s]\in\operatorname{Pic}(U_{k^s})[2]$.
	
	Let $X_0\to S$ be the minimal regular resolution over $k$.
	When $d=1$, the base point of $|-K_S|$ is a $k$-rational
	point lying in $U$, and hence lifts to $X_0$.
	Thus $X_0(k)\neq\varnothing$ in both cases.
	Therefore, the natural maps
	\[
	\operatorname{Br}(k)
	\longrightarrow\operatorname{Br}(X_0)
	\longrightarrow\operatorname{Br}(U)
	\]
	are injective.
	The Hochschild--Serre exact sequence
	\[
	0\longrightarrow\operatorname{Pic}(U)
	\longrightarrow\operatorname{Pic}(U_{k^s})^{G_k}
	\longrightarrow\operatorname{Br}(k)
	\longrightarrow\operatorname{Br}(U)
	\]
	shows that $\mathcal L_s$ descends to a line bundle
	$\mathcal L$ on $U$ satisfying
	\[
	\mathcal L^{\otimes2}\cong\mathcal O_U.
	\]
	The cyclic cover construction
	\cite[3.5]{EsnaultViehweg}
	gives a finite \'etale double cover $V\to U$.
	Since $\overline{\mathcal L}$ is nontrivial,
	this cover is geometrically integral.
	
	Let $T$ be the normalization of $S$ in $k(V)$.
	Then $T$ is normal and geometrically integral, and
	the natural morphism $\pi:T\to S$ is finite of degree $2$.
	It restricts to $V\to U$, so it is \'etale over $U$;
	it is quasi-\'etale since
	$S\setminus U$ has codimension two.
\end{proof}

\section{Unirationality}\label{sec:2}

\subsection{del Pezzo surface of degree $2$}
We collect some properties of generalized Eckardt points.

\begin{defi}
	A geometric point of a del Pezzo surface of degree two is
	called a \defterm{generalized Eckardt point} if it lies on
	four distinct $(-1)$-curves.
\end{defi}
%\begin{lemm}\label{lemm:eckardt}
%	Let $S$ be a del Pezzo surface of degree two over $\overline{k}$.
%	Every point of $S$ lies on at most four $(-1)$-curves.
%	If four distinct $(-1)$-curves $C_1,\ldots,C_4$ pass through a point $x$,
%	then they meet pairwise transversely at $x$ and have
%	no other intersections.
%	Moreover, $C_1+\cdots+C_4\sim-2K_S$.
%\end{lemm}
\begin{lemm}[\cite{Tamanoi}, Lemma~4.8 and its proof]\label{lemm:eckardt}
	Let $S$ be a weak del Pezzo surface of degree~$2$ over $\bar{k}$.
	Every point of $S$ lies on at most four $(-1)$-curves.
	If four distinct $(-1)$-curves $C_1,\ldots,C_4$ pass through a
	point $x$, then each pair meets transversely at $x$ and nowhere
	else, and $C_1+\cdots+C_4\sim -2K_S$.
\end{lemm}
%
%\begin{proof}
%	The bound follows from \cite[Lemma~2.2]{STVA14}.
%	If four $(-1)$-curves $C_1,\cdots,C_4$ pass through $x$, then $x$ lies
%	outside the ramification divisor of the anticanonical
%	morphism $\kappa:S\to\mathbb P^2$
%	by \cite[Example~2.4]{STVA14}.
%	These curves map isomorphically onto four distinct
%	lines through $\kappa(x)$.
%	Since $\kappa$ is \'etale at $x$, the curves $C_1,\cdots,C_4$ are pairwise transverse at $x$ and have no other intersections.
%	
%	Put $D=C_1+\cdots+C_4$. Since $C_i^2=-1$ and
%	$C_i\cdot C_j=1$ for $i\ne j$, we have
%	\[
%	D^2=8,\qquad (-K_S)\cdot D=4.
%	\]
%	As $K_S^2=2$, equality holds in the Hodge index inequality,
%	so $D\equiv-2K_S$. Since numerical and linear equivalence
%	coincide on the smooth rational surface $S$,
%	we obtain $D\sim-2K_S$.
%\end{proof}

We now prove Theorem~\ref{theo:main}.
Below is a commutative
diagram summarizing the unirationality construction.

\begin{equation}\label{diag}
	\begin{tikzcd}
		D \arrow[r,"\pi_C"] \arrow[d,hook]
		& C \arrow[d,hook]
		& E_x \arrow[l, phantom, "\simeq" description] \arrow[d,hook]
		& \\
		T \arrow[r,"\pi"]
		& S_2
		& S_1 \arrow[l,"\rho"'] \arrow[r,"\rho_1"]
		& S \\
		& Y \arrow[ul,"g"] \arrow[ur,"f"']
		& & &
	\end{tikzcd}
\end{equation}

\begin{proof}[Proof of Theorem~\ref{theo:main}]
	Let $x\in S(k)$ be a generalized Eckardt point, and let
	$C_1,\ldots,C_4$ be the four geometric $(-1)$-curves
	through $x$.
	Since there are no further $(-1)$-curves through $x$ by Lemma~\ref{lemm:eckardt}, the sum $\sum_iC_i$ is defined over $k$.

	Let $\rho_1:S_1\to S$ be the blow-up at $x$, with
	exceptional divisor $E_x$ and $\widetilde C_i$ the strict transform of $C_i$.
	If $A$ is any geometric irreducible curve on $S_1$
	other than $E_x$ and $\widetilde C_i$, then
	\[
	-2K_S\cdot\rho_1(A)
	=\sum_{i=1}^4(\widetilde C_i+E_x)\cdot A
	\ge 4E_x\cdot A.
	\]
	Therefore,
	\[
	(-K_{S_1})\cdot A
	=(-K_S)\cdot\rho_1(A)-E_x\cdot A
	\ge \frac12(-K_S)\cdot\rho_1(A)>0,
	\]
	where the last inequality follows from the ampleness
	of $-K_S$. Moreover,
	\[
	(-K_{S_1})\cdot E_x=1,
	\qquad
	(-K_{S_1})\cdot\widetilde C_i=0.
	\]
	Thus $-K_{S_1}$ is nef; it is big since $K_{S_1}^2=1$.
	Hence $S_1$ is a weak del Pezzo surface of degree one,
	whose only $(-2)$-curves are
	$\widetilde C_1,\ldots,\widetilde C_4$.

	The $1$-cycle $\sum_i\widetilde C_i$ is defined over $k$,
	and let $\rho:S_1\to S_2$ be the contraction of it.
	Then $S_2$ is a Du Val del Pezzo surface of degree $1$
	with four $A_1$ singularities.
	Since $E_x$ meets each $\widetilde C_i$ once,
	Proposition~\ref{prop:distinguished-curve} shows that
	$S_2$ is a generalized Iskovskikh surface of type $+$.
	Thus there is a quasi-\'etale double cover $\pi:T\ra S_2$
	defined over $k$ by Corollary~\ref{coro:cover}.
	
	Put $B=\sum_i\widetilde C_i$ as in the proof of Corollary~\ref{coro:cover}. Then
	$B\sim-2(K_{S_1}+E_x)$.
	Let $Y$ be the normalization of $S_1$ in $k(T)$,
	with induced morphisms $f:Y\to S_1$ and $g:Y\to T$.
	By the construction of $\pi$, the morphism $f$ is
	the double cover associated with
	$\mathcal O_{S_1}(-K_{S_1}-E_x)$ and branched along $B$.
	Since $B$ is geometrically a disjoint union of smooth curves,
	$Y$ is smooth.
	
	The curves $F_i=(f^{-1}\widetilde C_i)_{\mathrm{red}}$
	are disjoint $(-1)$-curves, and $g$ contracts precisely
	these curves. Thus $T$ is smooth.
	Since $\pi$ is quasi-\'etale, 
%$K_T=\pi^*K_{S_2}$,
	$T$ is a del Pezzo surface of degree $2$.

	Let $C\subset S_2$ be the image of $E_x$, and let
	$D\subset T$ be the reduced inverse image of $C$ under $\pi$.
	The curve $\widetilde D=f^{-1}(E_x)$ is smooth of genus $1$,
	since it is a double cover of $E_x\simeq\mathbb P^1_k$
	branched at four distinct geometric points.
	It meets each $F_i$ transversely once, so its image $D$
	is also smooth.
	By \eqref{eq:Z} and the canonical divisor formula,
	\[
	\widetilde D=f^*E_x\sim-K_Y,
	\qquad
	g^*D=\widetilde D+\sum_iF_i\sim-g^*K_T.
	\]
	Hence $D\in|-K_T|$.
	Let $R$ be the ramification divisor of the
	anticanonical morphism of $T$.
	Since $D$ is the reduced anticanonical member lying over a line,
	it is not contained in the ramification divisor. Hence $D\cap R$
	is finite.
	
	The curve $C^\circ=C\cap(S_2)_{\mathrm{sm}}$ is isomorphic
	to an open subset of $E_x\simeq\mathbb P^1_k$.
	Put $\pi_C=\pi|_D:D\to C$.
	Since $k$ is infinite and $T$ has only finitely many
	generalized Eckardt points, we may choose $p\in C^\circ(k)$
	such that both geometric points of $\pi_C^{-1}(p)$
	lie outside $R$ and are not generalized Eckardt points.
	
	Let $\sigma\in H^1(k,\mathbb Z/2\mathbb Z)$ be the class
	of the fibre $\pi^{-1}(p)$.
	Twisting $\pi$ by $\sigma$ gives a double cover
	$\pi^\sigma:T^\sigma\to S_2$ whose fibre above $p$ is split.
	In particular, it contains a $k$-rational point $q$.
	The anticanonical ramification divisor and the set
	of generalized Eckardt points are preserved by
	automorphisms, so $q$ satisfies the hypotheses of
	\cite[Corollary~3.3]{STVA14}.
	Thus $T^\sigma$ is $k$-unirational.
	Since $T^\sigma$ dominates $S_2$, which is $k$-birational
	to $S$, the surface $S$ is also $k$-unirational.
\end{proof}

\begin{proof}[Proof of Corollary~\ref{coro:unirationality}]
	If $k$ is finite, the assertion follows from
	\cite[Theorem~1.1]{FvL16}.
	Suppose that $k$ is infinite, and choose
	$P\in(S\setminus R)(k)$.
	If $P$ is a generalized Eckardt point, apply
	Theorem~\ref{theo:main}.
	Otherwise, apply \cite[Corollary~3.3]{STVA14}.
\end{proof}

\subsection{Du Val del Pezzo of degree $2$ and $1$}

\begin{proof}[Proof of Corollary~\ref{coro:rdp}]
	If $S$ is smooth, the assertion follows from Corollary \ref{coro:unirationality}. We may therefore assume that \(S\) is singular.
	If $S$ is $k$-unirational, then $S(k)$ is Zariski dense in $S$.
	Conversely, fix $p\in(S\setminus R)(k)$. By
	\cite[Theorem~1.2(2)]{Tamanoi}, we may assume that $p$ is a
	generalized Eckardt point. Choose coordinates over $k$ such that
	$S=\{w^2=F(x,y,z)\}$ and $p=(0:0:1:1)$.
	
	The divisor \(C_1+\cdots+C_4\in|-2K_S|\)
	is cut out by a weighted quadratic form \(w=Q(x,y,z)\) over $\bar k$.
	Thus $F-Q^2=b(x,y)$ for a squarefree quartic $b$.
	Since \(Q(0,0,1)=1\), comparison of the coefficients of $z^3$ and $z^2$ in $F$ and $Q^2$ shows that
	$Q$ is defined over $k$. Completing the square in $z$, we obtain
	\[
	w^2=(z^2+h)^2+b=z^4+2hz^2+g,
	\qquad g=h^2+b,
	\]
	with $h,b,g\in k[x,y]$.

	Since $b$ is squarefree, the Jacobian criterion shows that the singular points of \(S_{\bar k}\) occur precisely
	at $z=w=0$ above the multiple roots of $g$. At a root of multiplicity $m\ge2$, we have $h\ne0$, since otherwise $b=g-h^2$
	would also have a multiple root.
	Hence $2h+z^2$ is a unit in $
	\bar{k}[\![t,z]\!]$.
	A direct computation shows that $S_{\overline{k}}$ can be written as
	$uv=t^m$ in the completed local ring.
	Each root of $g$ of multiplicity $m\ge2$ gives an $A_{m-1}$ singularity. Thus the only possible singularity types of \(S_{\bar k}\) are $A_1$, $2A_1$, $A_2$, and $A_3$.

	Let $X=\operatorname{Bl}_p S$, let $E$ be its exceptional curve,
	and let $S_2$ be its anticanonical model. In characteristic zero,
	the classification in \cite[Theorem~4.1]{gao2024quasi}
	shows that $(S_2)_{\bar k}$ admits a quasi-\'etale double cover.
	To construct such a cover over $k$ in any characteristic different
	from $2$, let $B$ be the strict transform of $\{w=z^2+h\}$ on $X$.
	Then $B$ is defined over $k$ and $B\sim -2(K_X+E)$ over $k$.
	As in the proof of Theorem~\ref{theo:main}, taking the double cover
	of $X$ branched along $B$ and then the Stein factorization of its
	composite with $X\to S_2$ gives a quasi-\'etale double cover
	$\pi:T\to S_2$ over $k$.
	The surface \(T\) is smooth above the four new nodes, and \(\pi\) is \'etale above the original singularities.
	So $T_{\bar k}$ has type $2A_1$, $4A_1$, $2A_2$, or $2A_3$, respectively.
	
	Since the curve $D$ above the image of $E$ is not contained in the anticanonical ramification locus,
	the twisting argument in the proof of Theorem~\ref{theo:main}
	now applies and \cite[Theorem~1.2(2)]{Tamanoi} shows that a twist $T^\sigma$ of $\pi$ is $k$-unirational.
	Thus $S$ is $k$-unirational.
\end{proof}

\begin{proof}[Proof of Corollary~\ref{coro:rdpdp1}]
	The base point of $|-K_S|$ is a smooth $k$-point. By the proof of \cite[Proposition~A.1(i)]{prokhorov2026} (the argument is given in characteristic zero, but the same proof
	works over any perfect field of characteristic different from~$2$), it suffices to consider the case where the minimal resolution of $S$ is the blow-up of a smooth degree-two del Pezzo surface $V$ at $p\in V(k)$. If $p$ lies outside the ramification divisor, Corollary \ref{coro:unirationality} applies. Otherwise, by \cite[Theorem 2.9]{STVA14}, the geometric $(-2)$-curves are the strict transforms of the components above the tangent line to the branch quartic at the image of $p$. They form either a single $(-2)$-curve or two $(-2)$-curves meeting once, giving type $A_1$ or $A_2$, respectively.
\end{proof}

\section{Further applications}
\label{sec:3}
\subsection{Bigonal duality}
We first work over $\bC$.
Let $\Gamma\subset\bP^2$ be a smooth plane quartic
with a bielliptic involution $\tau$, and let
$\rho:\Gamma\to E=\Gamma/\langle\tau\rangle$
be the quotient map.
Projection from the isolated fixed point of the
induced involution of $\bP^2$ gives a double cover
$\mu:E\to\bP^1$.
Pantazis's bigonal duality is the following transformation
of bielliptic quartics:
\begin{equation}\label{eq:bigonal}
	(\Gamma,\tau)\longmapsto
	\left(
	\widehat{\Gamma}
	:=
	\{P+Q\in \operatorname{Div}^{(2)}(\Gamma)\mid
	\mathcal O_E(\rho(P)+\rho(Q))
	\simeq\mu^*\mathcal O_{\mathbb P^1}(1)\},
	\ \tau^*|_{\widehat{\Gamma}}
	\right),
\end{equation}
see \cite[proof of Proposition~2.1]{Menet}.
Here $\operatorname{Div}^{(2)}(\Gamma)$ denotes
the space of effective divisors of degree $2$ on $\Gamma$.
The following proposition identifies this transformation
with the construction in \eqref{diag}.

\begin{prop}
	\label{prop:inverse-construction}
	Let $S$ be a smooth del Pezzo surface of degree $2$
	over $\bC$ with a generalized Eckardt point $p$,
	and let $T$ be obtained from $(S,p)$ by the
	construction in \eqref{diag}.
	Then $T$ also admits a generalized Eckardt point.
	Moreover, the bigonal transformation \eqref{eq:bigonal}
	sends the anticanonical branch quartic of $T$,
	with the involution induced by $T\to S_2$,
	to that of $S$, with the bielliptic involution
	determined by $p$.
\end{prop}

\begin{proof}
	Let $\iota$ be the deck involution of
	$\pi:T\to S_2$.
	Since $\iota$ is distinct from the Geiser involution,
	it induces a nontrivial involution $\tau$ on the
	branch quartic $\Gamma$ of $T$.
	Thus $\Gamma$ is bielliptic, so $T$ admits a
	generalized Eckardt point by the classical result of
	Kowalevskaya \cite{Kowalevski1884}
	(see also \cite[Theorem~3.6]{JPZ25}).
	As in \cite[Example~2.4]{STVA14},
	we may choose weighted homogeneous coordinates
	in $\bP(1,1,1,2)$
	such that
	\begin{equation}\label{eq:iota}
		T=V(w^2+z^4+z^2q_2(x,y)+q_4(x,y)),
		\quad
		\iota(x:y:z:w)=(x:y:-z:-w).
	\end{equation}
	Putting $a=z^2$ and $b=zw$, we obtain
	\begin{equation}\label{eq:S2}
		S_2=V(b^2+a^3+q_2a^2+q_4a)\subset
		\bP(1,1,2,3).
	\end{equation}
	The fixed locus of $\iota$ consists of the four
	points $\{z=w=0\}$, whose images are precisely
	the four $A_1$ singularities of $S_2$.
	
	Let $\rho:S_1\to S_2$ be the minimal resolution,
	with exceptional curves $A_1,\ldots,A_4$,
	and let $E$ be the strict transform of
	$C=\{a=b=0\}$.
	Then $E\simeq\bP^1$, $E\cdot A_i=1$, and
	\[
	2E+\sum_iA_i\sim-2K_{S_1}.
	\]
	Hence $E^2=-1$.
	By the uniqueness in
	Proposition~\ref{prop:distinguished-curve},
	$E$ is the exceptional curve of the original
	blow-up $\rho_1:S_1\to S$ at $p$.
	
	Put $r=b/a$.
	As a rational section of $\mathcal O_{S_1}(-K_{S_1})$,
	$r$ has polar divisor $E$.
	Since
	\(
	\rho_1^*(-K_S)=-K_{S_1}+E,
	\)
	the three linearly independent sections $r,x,y$
	form a basis of $H^0(S_1,-K_{S_1}+E)$
	and give the anticanonical map of $S$.
	Substituting $b=ar$ into the equation of $S_2$
	and dividing by $a$, we obtain
	\[
	a^2+(q_2+r^2)a+q_4=0.
	\]
	Thus, putting $W=2a+q_2+r^2$, we obtain
	the anticanonical model
	\begin{equation}\label{eq:anticanonical}
		S:\quad W^2=(r^2+q_2)^2-4q_4.
	\end{equation}
	
	On the chart $y=1$, let $P+Q$ be a general point
	of $\widehat{\Gamma}$.
	By its definition, $z(P)^2$ and $z(Q)^2$
	are the two roots of
	$a^2+q_2a+q_4=0$.
	Vieta's formulas give
	\[
	z(P)^2+z(Q)^2=-q_2,
	\qquad
	z(P)^2z(Q)^2=q_4.
	\]
	Thus $t=z(P)+z(Q)$ gives the plane model
	\[
	\widehat{\Gamma}:\quad (t^2+q_2)^2=4q_4.
	\]
	The unordered pair $\{z(P),z(Q)\}$ is recovered
	from $t$ and $(t^2+q_2)/2$.
	This identifies $\widehat{\Gamma}$ with the
	branch quartic of \eqref{eq:anticanonical}.
	Under this identification, $\tau^*$ acts by
	$(r:x:y)\mapsto(-r:x:y)$.
	Since $x$ and $y$ vanish along $E$, the
	anticanonical image of $p$ is $(1:0:0)$ (also see \cite[Example~2.4]{STVA14}).
	Thus $\tau^*$ is the bielliptic involution determined by $p$.
\end{proof}

\begin{rema}
	Under \eqref{eq:iota}, the Kowalevskaya involution defined in \cite[\S 6]{CMYZ26} is given by $\kappa:z\mapsto -z$, and the Geiser involution is given by $\sigma:w\mapsto -w$.
	Thus our \'etale-in-codimension-one involution $\iota$ defining $\pi:T\to S_2$ can be written as $\iota=\kappa\circ\sigma$.
\end{rema}

\subsection{Real plane Cremona group}
We now work over $\mathbb R$ and use the term Kowalevskaya
involution in the sense of \cite{CMYZ26}.
By \cite[Remark~6.6 and Proposition~6.8]{CMYZ26}, an involution
$\kappa$ of a smooth real del Pezzo surface $T$ of degree~$2$ is
a Kowalevskaya involution if and only if $T$ is $\mathbb R$-rational,
$\operatorname{rk}\operatorname{Pic}(T)^{\langle\kappa\rangle}=1$,
and $\kappa$ is distinct from the Geiser involution.
The following gives a negative answer to
\cite[Question~6.9]{CMYZ26}.

\begin{prop}\label{prop:kowalevskaya}
	Let $D$ be the fixed curve of a Kowalevskaya involution.
	Then uncountably many pairwise nonconjugate Kowalevskaya
	involutions in $\operatorname{Bir}(\mathbb P^2_{\mathbb R})$
	have fixed curve isomorphic to $D$ over $\mathbb R$.
\end{prop}

\begin{proof}
	Let $(T_0,\kappa)$ be a Kowalevskaya pair with fixed curve $D$.
	In the real normal form of \cite[p.~3680]{CMYZ26}, write
	\[
	T_0:\quad -w^2=z^4+q_2(x,y)z^2+q_4(x,y),
	\qquad \kappa(x:y:z:w)=(x:y:-z:w).
	\]
	Put $q_{2,t}=q_2+txy$, and let $T_t$ be obtained by replacing
	$q_2$ with $q_{2,t}$. For sufficiently small real $t$, the
	surface $T_t$ is smooth and $\Delta_t=q_{2,t}^2-4q_4$ still has four
	distinct real roots. By local triviality of the smooth proper
	family, $T_t(\mathbb R)$ is diffeomorphic to $T_0(\mathbb R)$,
	so $T_t$ is $\mathbb R$-rational by
	\cite[Theorem~2.2]{CMYZ26}.
	Completing the square in $z^2$ and normalizing the four roots
	of $\Delta_t$ puts the equation in the form of
	\cite[Remark~6.6]{CMYZ26}. Thus
	$\operatorname{Pic}(T_t)^{\langle\kappa\rangle}$ has rank one.
	Since $\kappa$ acts nontrivially on the anticanonical plane,
	\cite[Proposition~6.8]{CMYZ26} shows that it is a Kowalevskaya
	involution. Its fixed curve is the same curve
	$D=\{z=0,\ w^2=-q_4(x,y)\}$ for every such $t$.
	
	Let $\sigma$ denote the Geiser involution of $T_t$.
	As in \eqref{eq:S2}, we set
	$a=z^2$ and $b=zw$. Then the quotient
	$S_{2,t}=T_t/\langle\kappa\sigma\rangle$ is an Iskovskikh surface of degree $1$ and type $4A_1+$ given by
	\[
	b^2=-a(a^2+q_{2,t}a+q_4).
	\]
	Its Bertini involution is $b\mapsto-b$. The only dimension-one components of its fixed
	locus are the distinguished curve $C=\{a=b=0\}$ and the
	genus-one curve
	\[
	E_t:\quad \Delta_t=(2a+q_2+txy)^2.
	\]
	Thus the $j$-invariant $j(E_t)$ is determined by the isomorphism class of $S_{2,t}$.
	As $t\to\infty$,
	the binary quartic $t^{-2}\Delta_t$ tends to $x^2y^2$.
	Its four roots therefore degenerate in two pairs at distinct
	points. After ordering the roots accordingly, their cross-ratio
	tends to zero, so $j(E_t)$ tends to infinity \cite[p.32]{Mukai}. Thus $j(E_t)$ is a nonconstant rational
	function of $t$. Hence it takes
	uncountably many distinct values on every nonempty
	open interval in $\bR$.
	
	By \cite[Theorem~2.9 and discussion after
	Proposition~6.8]{CMYZ26}, conjugate Kowalevskaya involutions
	have equivariantly isomorphic del Pezzo models. Since the
	Geiser involution is intrinsic, such an isomorphism descends
	to the quotients $S_{2,t}$. Distinct values of $j(E_t)$
	therefore give nonconjugate involutions.
\end{proof}
\subsection{Equivariant unirationality}
\begin{prop}\label{prop:equivariant-unirationality}
	Let $S$ be a smooth del Pezzo surface of degree~$2$ over an infinite
	field $k$ of characteristic different from~$2$, and let $R$ be its
	ramification divisor. If a finite group $G$ acts on $S$ over $k$
	and fixes a point of $(S\setminus R)(k)$, then $S$ is $G$-unirational.
\end{prop}

\begin{proof}
	Let $K/k$ be any field extension and let $P\to\operatorname{Spec}K$
	be any $G_K$-torsor. Write ${}^{P}S$ for the twist of $S_K$ by $P$.
	Since $p\in(S\setminus R)(k)$ is fixed by $G$, twisting the
	$G_K$-equivariant map $p_K\colon\operatorname{Spec}K\to S_K\backslash R_K$
	gives a $K$-point of ${}^{P}S$ outside its ramification divisor.
	Thus ${}^{P}S$ is $K$-unirational by Corollary~\ref{coro:unirationality}.
	As this holds for every $K$ and $P$, \cite[Theorem~1.1(c)]{DR}
	implies that $S$ is $G$-unirational.
\end{proof}

\begin{rema}\label{rema:noncyclic}
	For a noncyclic group $G$, a minimal complex del Pezzo $G$-surface
	of degree two with a fixed point $p$ can only have two types (see \cite[Theorem~1.1]{DD}):
	\begin{itemize}
		\item (2A) $p$ lies on the ramification divisor $R$;
		\item (2B) $p$ is a generalized Eckardt point.
	\end{itemize}
	To our knowledge, Proposition~\ref{prop:equivariant-unirationality}
	gives the first proof of $G$-unirationality in Case~$(2\mathrm B)$
	in full generality. Whether every surface in Case~$(2\mathrm A)$
	is $G$-unirational remains open.
\end{rema}
\begin{exam}
	The following example shows how equivariant unirationality can depend
	on the choice of lift of an action on the anticanonical plane.
	Consider the degree-two del Pezzo surface
	\[
	S_a=\{w^2=x^4+y^4+ax^2y^2+z^4\}
	\subset \mathbb P_{\mathbb C}(1,1,1,2), \qquad a^2\ne 4.
	\]
	Let $Q_8$ act on $(x,y)$ via
	$A(x,y)=(ix,-iy)$ and $B(x,y)=(-y,x)$, fixing $z$ and $w$.
	The point $(0:0:1:1)$ is fixed and lies outside the ramification
	divisor, so $S_a$ is $Q_8$-unirational by
	Proposition~\ref{prop:equivariant-unirationality}.
	Replacing $A$ by $\sigma A$, where $\sigma$ is the Geiser involution,
	gives the lift studied by Tschinkel and Zhang~\cite{TZ},
	who prove that $S_a$ is not $Q_8$-unirational
	for this action.
\end{exam}

\normalem
\bibliographystyle{alpha}
\bibliography{myref25}

\begin{thebibliography}{CMYZ26}

\bibitem[CLOP26]{cheltsov2026unirational}
Ivan Cheltsov, Konstantin Loginov, Dmitri Orlov, and Yuri Prokhorov.
\newblock Unirational del pezzo surfaces of degree one.
\newblock {\em arXiv preprint arXiv:2608.15435}, 2026.

\bibitem[CMYZ26]{CMYZ26}
Ivan Cheltsov, Fr{\'e}d{\'e}ric Mangolte, Egor Yasinsky, and Susanna
  Zimmermann.
\newblock Birational involutions of the real projective plane.
\newblock {\em J. Eur. Math. Soc. (JEMS)}, 28(8):3653--3711, 2026.

\bibitem[CP21]{CheltsovProkhorov2021}
Ivan Cheltsov and Yuri Prokhorov.
\newblock {Del Pezzo} surfaces with infinite automorphism groups.
\newblock {\em Algebraic Geometry}, 8(3):319--357, 2021.

\bibitem[CT88]{CorayTsfasman1988}
D.~F. Coray and M.~A. Tsfasman.
\newblock Arithmetic on singular {Del Pezzo} surfaces.
\newblock {\em Proceedings of the London Mathematical Society}, 57(1):25--87,
  1988.

\bibitem[DD16]{DD}
Igor Dolgachev and Alexander Duncan.
\newblock Fixed points of a finite subgroup of the plane {Cremona} group.
\newblock {\em Algebr. Geom.}, 3(4):441--460, 2016.

\bibitem[Dem80]{Demazure}
Michel Demazure.
\newblock Surfaces de {Del Pezzo}, {I--V}.
\newblock In Michel Demazure, Henry~Charles Pinkham, and Bernard Teissier,
  editors, {\em S{\'e}minaire sur les Singularit{\'e}s des Surfaces}, volume
  777 of {\em Lecture Notes in Mathematics}, pages 21--69. Springer-Verlag,
  Berlin, 1980.

\bibitem[DJT08]{DerenthalJoyceTeitler}
Ulrich Derenthal, Michael Joyce, and Zachariah Teitler.
\newblock The nef cone volume of generalized {Del Pezzo} surfaces.
\newblock {\em Algebra \& Number Theory}, 2(2):157--182, 2008.

\bibitem[DM10]{DynkinMinchenko2010}
E.~B. Dynkin and A.~N. Minchenko.
\newblock Enhanced {Dynkin} diagrams and {Weyl} orbits.
\newblock {\em Transformation Groups}, 15(4):813--841, 2010.

\bibitem[DR15]{DR}
Alexander Duncan and Zinovy Reichstein.
\newblock Versality of algebraic group actions and rational points on twisted
  varieties.
\newblock {\em J. Algebraic Geom.}, 24(3):499--530, 2015.

\bibitem[DSW24]{DSW2024}
Julian Demeio, Sam Streeter, and Rosa Winter.
\newblock Weak weak approximation and the {Hilbert} property for degree 2 {del
  Pezzo} surfaces.
\newblock {\em Proc. Lond. Math. Soc. (3)}, 128(5):e12601, 2024.

\bibitem[EV92]{EsnaultViehweg}
H{\'e}l{\`e}ne Esnault and Eckart Viehweg.
\newblock {\em Lectures on Vanishing Theorems}, volume~20 of {\em DMV Seminar}.
\newblock Birkh{\"a}user Verlag, Basel, 1992.

\bibitem[FvL16]{FvL16}
Dino Festi and Ronald van Luijk.
\newblock Unirationality of {del Pezzo} surfaces of degree 2 over finite
  fields.
\newblock {\em Bulletin of the London Mathematical Society}, 48(1):135--140,
  2016.

\bibitem[Gao24]{gao2024quasi}
Runxuan Gao.
\newblock Quasi-\'etale covers of {D}u {V}al del {P}ezzo surfaces and {Z}ariski
  dense exceptional sets in {M}anin's conjecture.
\newblock {\em arXiv preprint arXiv:2406.16240}, 2024.

\bibitem[JPZ25]{JPZ25}
Marek Janasz, Piotr Pokora, and Marcin Zieli{\'n}ski.
\newblock On arrangements of smooth plane quartics and their bitangents.
\newblock {\em Fundamenta Mathematicae}, 268:151--170, 2025.

\bibitem[Kow84]{Kowalevski1884}
Sophie Kowalevski.
\newblock {\"Uber die Reduction einer bestimmten Klasse Abel'scher Integrale
  $3^{\mathrm{ten}}$ Ranges auf elliptische Integrale}.
\newblock {\em Acta Mathematica}, 4:393--414, 1884.

\bibitem[Man86]{manin1986cubicforms}
Yu~I Manin.
\newblock {\em Cubic forms: algebra, geometry, arithmetic}, volume~4.
\newblock Elsevier, 1986.

\bibitem[Men14]{Menet}
Gr{\'e}goire Menet.
\newblock Duality for relative {Prymians} associated to {K3} double covers of
  {del Pezzo} surfaces of degree 2.
\newblock {\em Mathematische Zeitschrift}, 277(3--4):893--907, 2014.

\bibitem[Muk03]{Mukai}
Shigeru Mukai.
\newblock {\em An Introduction to Invariants and Moduli}, volume~81 of {\em
  Cambridge Studies in Advanced Mathematics}.
\newblock Cambridge University Press, Cambridge, 2003.
\newblock Translated by W. M. Oxbury.

\bibitem[Pro26]{prokhorov2026}
Yuri Prokhorov.
\newblock Double veronese cones with singularities, 2026.

\bibitem[{Sta}18]{stacks-project}
The {Stacks Project Authors}.
\newblock \textit{Stacks Project}.
\newblock \url{https://stacks.math.columbia.edu}, 2018.

\bibitem[STVA14]{STVA14}
Cec{\'i}lia Salgado, Damiano Testa, and Anthony V{\'a}rilly-Alvarado.
\newblock On the unirationality of {del Pezzo} surfaces of degree 2.
\newblock {\em J. Lond. Math. Soc. (2)}, 90(1):121--139, 2014.

\bibitem[Tam22]{Tamanoi}
Ryota Tamanoi.
\newblock Unirationality of {RDP} {Del Pezzo} surfaces of degree 2.
\newblock {\em J. Algebra}, 606:471--518, 2022.

\bibitem[Tre19]{Tre}
Andrey Trepalin.
\newblock Quotients of del {Pezzo} surfaces.
\newblock {\em Internat. J. Math.}, 30(12):1950068, 2019.

\bibitem[TZ26]{TZ}
Yuri Tschinkel and Zhijia Zhang.
\newblock Cohomological obstructions to equivariant unirationality.
\newblock {\em Int. Math. Res. Not. IMRN}, 2026(15), 2026.
\newblock Article No. rnag166.

\end{thebibliography}

\end{document}